\documentclass{amsart}

\usepackage{amsmath}
\usepackage{amssymb}
\usepackage{amsthm}
\usepackage{amsfonts}
\usepackage{calligra}
\usepackage{appendix}
\usepackage{graphicx}
\usepackage{enumerate}
\usepackage{mathrsfs}
\usepackage[english]{babel}
\usepackage{mathtools}
\usepackage{verbatim}
\usepackage{color}

\def\R{\mathbb R}

\def\N{\mathbb N}

\def\T{\mathbb{T}}
\def\eps{\varepsilon}

\newcommand{\Id}{\textrm{Id}}
\renewcommand{\div}{\textrm{div}}
\newcommand{\curl}{\textrm{curl}}
\newcommand{\dist}{\textrm{dist}}

\newcommand{\abs}[1]{\left\lvert#1\right\rvert}

\newcommand{\defeq}{\mathrel{:\mkern-0.25mu=}}
\newcommand{\eqdef}{\mathrel{=\mkern-0.25mu:}}

\newtheorem{theorem}{Theorem}
\newtheorem{lemma}{Lemma}
\newtheorem{proposition}{Proposition}

\newtheorem{corollary}{Corollary}

\theoremstyle{definition}

\theoremstyle{remark}
\newtheorem{remark}{Remark}

\begin{document}

\title{Magnetic helicity, weak solutions and relaxation of ideal MHD}

\subjclass[2010]{35Q35, 76W05, 76B03}

\keywords{Magnetohydrodynamics, convex integration, conservation laws, compensated compactness}
\date{}

\author{Daniel Faraco}
\address{Departamento de Matem\'{a}ticas \\ Universidad Aut\'{o}noma de Madrid, E-28049 Madrid, Spain; ICMAT CSIC-UAM-UC3M-UCM, E-28049 Madrid, Spain}
\email{daniel.faraco@uam.es}
\thanks{D.F. acknowledges financial support from the Spanish Ministry of Science and Innovation through the Severo Ochoa Programme for Centres of Excellence in R\&D(CEX2019-000904-S and by the MTM2017-85934-C3-2-P2. He was also partially supported by CAM through the Line of excellence for University Teaching Staff between CM and UAM. D.F. was also partially supported by the ERC Advanced Grant 834728. S.L. was supported by the ERC Advanced Grant 834728. L. Sz. was supported by ERC Grant 724298-DIFFINCL.}

\author{Sauli Lindberg}
\address{Department of Mathematics and Statistics \\ University of Helsinki, P.O. Box 68, 00014 Helsingin yliopisto, Finland}
\email{sauli.lindberg@helsinki.fi}

\author{L\'{a}szl\'{o} Sz\'{e}kelyhidi, Jr.}
\address{Institut f\"{u}r mathematik \\ Universit\"{a}t Leipzig, Augustusplatz 10, D-04109, Leipzig,
Germany}
\email{laszlo.szekelyhidi@math.uni-leipzig.de}

\date{}

\begin{abstract}
We revisit the issue of conservation of magnetic helicity and the Woltjer-Taylor relaxation theory in magnetohydrodynamics in the context of weak solutions. We introduce a relaxed system for the ideal MHD system, which decouples the effects of hydrodynamic turbulence such as the appearance of a Reynolds stress term from the magnetic helicity conservation in a manner consistent with observations in plasma turbulence. As by-products we
answer two open questions in the field: We show the sharpness of the $L^3$ integrability condition for magnetic helicity conservation and provide turbulent bounded solutions for MHD dissipating energy and cross helicity but with (arbitrary) constant magnetic helicity.

\end{abstract}

\maketitle

\section{Introduction}

In this paper we consider the system of ideal magnetohydrodynamics (MHD in short), which couples the incompressible Euler equations with the Faraday-Maxwell system via Ohm's law. The MHD system, with nonzero viscosity and magnetic resistivity, is used in modelling electrically conducting fluids such as plasmas and liquid metals (see~\cite{GLBL} and~\cite{ST}). The ideal MHD system, where kinematic viscosity and magnetic diffusivity are set to zero, contains a wealth of mathematical structure \cite{ArnoldKhesinbook} and can be written as
\begin{equation}\label{e:MHD}
\begin{split}
 \partial_t u + u \cdot \nabla u- B \cdot \nabla B + \nabla p &= 0, \\
 \partial_t B + u \cdot \nabla B- B \cdot \nabla u &= 0,\\
 \nabla \cdot u = \nabla \cdot B &= 0, 
  \end{split}
  \end{equation}
Moreover, in analogy with the role of the incompressible Euler equations for hydrodynamical turbulence \cite{Onsager,CET}, the ideal system is relevant in the inviscid, irresitive ``turbulent'' limit in the context of weak solutions \cite{CKS,Eyi3}. 

A key question concerning weak solutions is to understand the correct space in which to formulate the problem. This question is closely related to the issue of anomalous dissipation and conservation of energy. Let us recall the conserved quantities. It is well known that energy and cross helicity are conserved by smooth solutions whereas magnetic helicity is preserved by turbulent solutions
(see section \ref{s:Conserved quantities} for the precise function spaces). To avoid technical issues concerning the topology of the domain and boundary conditions, we will work in the 3D periodic setting $\T^3$. Identical arguments are valid for simply connected magnetically closed domains. For a definition of magnetic helicity in domains with non-trivial topology see \cite{MV}.

\subsection{Weak solutions}

The ideal MHD system is obtained by combining the Euler system
for ideal incompressible fluids driven by a magnetic field 
\begin{subequations}\label{e:MHD1}
\begin{align}
 \partial_t u + \div (u \otimes u- B\otimes B +  p \, \Id) &= 0,  \\
 \nabla \cdot u&=0
 \end{align}
 \end{subequations}
 with the Faraday-Maxwell system
 \begin{subequations}\label{e:MHD2}
 \begin{align}
 \partial_t B + \nabla \times E &= 0, \label{e:MHD21}\\
 \nabla \cdot B &= 0, \label{e:MHD22}
\end{align}
\end{subequations}
with constitutive law given by Ohm's law for perfectly conducting fluids
 \begin{equation}\label{e:MHD3}
 E=B\times u.
 \end{equation}
In the formulation \eqref{e:MHD1}-\eqref{e:MHD3} it is apparent that solutions in the sense of distributions can be defined for $u,B\in L^2_{loc}$.

\subsection{Conserved quantities} \label{s:Conserved quantities}

Formally, the Euler system \eqref{e:MHD1} together with \eqref{e:MHD22} leads to the balance equations
\begin{subequations}\label{e:localconservation1}
\begin{align}
 \tfrac12\partial_t \Bigl(|u|^2+|B|^2\Bigr) + \div\Bigl(u \bigl(p+\tfrac12(|u|^2+|B|^2)\bigr) - B(u\cdot B)\Bigr) &= 0,  \\
 \partial_t \Bigl(u\cdot B\Bigr) + \div\Bigl(B \bigl(p-\tfrac12(|u|^2+|B|^2)\bigr) + u(u\cdot B)\Bigr) &= 0.
  \end{align}
  \end{subequations}
In integrated form these imply conservation of total energy and cross-helicity, where

\begin{subequations}
	\begin{align}
	\mathcal{E}(u,B)&=\frac{1}{2}\int_{\T^3}|u|^2+|B|^2\,dx,\\	
	\mathcal{W}(u,B)&=\int_{\T^3}u\cdot B\,dx.
	\end{align}
\end{subequations}

Next, define the magnetic vector potential $A=\curl^{-1}B$, defined uniquely on $\T^3$ by applying the Biot-Savart law to the system 
\begin{subequations}
\begin{align}
\curl A&=B,\label{e:magneticpotential}\\
\div A&=0.	
\end{align}
\end{subequations}

Then formally the Faraday-Maxwell system \eqref{e:MHD2} leads to 
\begin{equation}\label{e:localconservation2}
\partial_t\Bigl( A\cdot B\Bigr)+\div\Bigl(E\times A-f B\Bigr)=-2B\cdot E,	
\end{equation}
where $f$ is a scalar function acting as ``electric scalar potential'' in the sense that from \eqref{e:MHD21} we obtain
\begin{equation}\label{e:electricpotential}
\partial_tA+E=\nabla f,
\end{equation}
see e.g \cite[Lemma 2.3]{FLS}. Using \eqref{e:MHD3}, from \eqref{e:localconservation2} we deduce conservation of magnetic helicity
$$
\mathcal{H}(B):=\int_{\T^3}A\cdot B\,dx.
$$
Indeed, in  manifolds with no boundary helicity is a gauge invariant property of the $n-1$ form induced by $B$--see the book \cite{ArnoldKhesinbook} and  \cite{Arn} for the precise  topological meaning of helicity in relation with the asymptotic Hopf invariant. 

Concerning weak solutions, an analogous development to the pure hydrodynamic case ($B=0$) has led to the following Onsager-type criteria for conservation, formulated in terms of spatial Besov spaces: in~\cite{CKS} is is shown that energy and cross helicity are conserved as long as $u,B \in C([0,T[; B^\alpha_{3,\infty})$ for $\alpha > 1/3$, and the end-point result $u,B \in L^3_t B^{1/3}_{3,c_0}$ is shown in~\cite{KL} by closely following~\cite{CCFS}. Magnetic helicity, in contrast, is already conserved if $u,B \in C([0,T];B^\alpha_{3,\infty})$ for $\alpha > 0$~\cite{CKS} or $u,B \in L^3(\T^3 \times [0,T])$ as shown by Kang and Lee in~\cite{KL}. There is also a wealth of function spaces in which the regularity is distributed differently in  $u,B$ for which cross helicity or energy is preserved.  Such conservation results lead to the natural flexible side of Onsager-type conjectures for H\"{o}lder continuous solutions, see for example \cite[Conjecture 2.10]{BVreview}. In this regard our Corollary 4 below shows the flexibility in the realm of $L^\infty$ solutions (for previous works see  \cite{BLFNL} and \cite{FLS} for null magnetic helicity). The remarkable robustness of magnetic helicity as a conserved quantity is reflected in simulations and experiments, as we next briefly discuss.

\subsection{Magnetic helicity, Woltjer-Taylor relaxation and
magnetic reconnection}

Woltjer proposed magnetic helicity conservation as an explanation of the observation that various astrophysical plasmas tend to evolve toward a force-free state $\nabla \times B = \alpha B$~\cite{Woltjer}. A similar relaxation process also occurs in many laboratory settings, even if the inital state of the system is turbulent~\cite{OS}. Woltjer suggested the variational problem of minimizing total energy under the constraint that magnetic helicity is fixed, and he computed formally that the minimizers are, indeed, force-free (see also~\cite{LA}). Moffatt~\cite{Mof} interpreted magnetic helicity topologically and noted that, furthermore, subhelicities over magnetically closed Lagrangian subvolumes are conserved (by smooth solutions). This abundance of conserved quantities is, however, seemingly at odds with the observed relaxation of turbulent plasmas~\cite{OS}.

Nevertheless, an important aspect of ideal MHD, or MHD with very low resistivity, is that the regime is consistent with turbulence and in particular magnetic reconnection can cause the non-conservation of subhelicities. Taylor conjectured in~\cite{Taylor} that in the presence of slight resistivity, subhelicity conservation would break down but magnetic helicity in the whole domain would nevertheless be approximately preserved (see~\cite{Berger,FL,FLMV} for 
mathematical confirmation of the latter hypothesis).

Taylor's ensuing relaxation theory (an archetypical
example of self-organization \cite{Has85}) revisits Woltjer's variational problem for magnetic energy emphasizing that magnetic reconnection would be responsible for energy dissipation but  magnetic helicity should be kept fixed. It is a matter of discussion in the physics literature (\cite{Eyi3} for example) what patterns of the MHD equation should prevail in the macroscopic variables compatible with magnetic reconnection and indeed under which circumstances Taylor relaxation theory is valid. We look at this issue from the perspective of mathematical relaxation, which we next describe.

\subsection{Relaxation}
 In the context of nonlinear PDE arising in continuum physics, mathematical relaxation has become an indispensable tool in large part due to the pioneering work of L.~Tartar in the 1970-80s (e.g. \cite{Tartar, Tartar2}). A key point in Tartar's program is to study the behaviour of (in general nonlinear) constitutive relations under weak convergence in combination with differential constraints arising from conservation laws. Weak limits can be interpreted as a deterministic analogue of averaging or coarse-graining, and thus, in many cases of interest, one is able to obtain 'averaged' constitutive relations. An indispensable and powerful tool in this program is compensated compactness. 
 
A standard example, treated for instance in \cite{Tartar2}, is the Maxwell system of electromagnetism. In particular for the Faraday-Maxwell system (Maxwell equations in vacuum) it is shown that $B\cdot E$ is a weakly continuous quantity where $E,B$ have the natural integrability conditions. A particularly elegant way of seeing this is by using space-time differential forms - Tartar attributes this observation to J. Robbin: The Faraday-Maxwell system \eqref{e:MHD2} can be equivalently formulated for the Faraday 2-form (mistakenly called the Maxwell 2-form in~\cite{FLS}) $\omega\in \Lambda^2(\R^4)$, related to the magnetic and electric fields $B,E$ via 
\begin{equation}\label{e:F2form}
\begin{split}
	\omega=B_1&dx_2\wedge dx_3+B_2dx_3\wedge dx_1+B_3dx_1\wedge dx_2\\
	&+E_1dx_1\wedge dt+E_2dx_2\wedge dt+E_3dx_3\wedge dt
\end{split}
\end{equation}
as
\begin{equation}\label{e:F1}
d\omega=0.	
\end{equation}
The corresponding potential 1-form $\alpha$ can be written as
\begin{equation}\label{e:F1form}
	\alpha=A_1 dx_1+A_2dx_2+A_3dx_3+fdt,
\end{equation}
so that $\omega=d\alpha$ is equivalent to \eqref{e:magneticpotential} and \eqref{e:electricpotential}. 

In this formalism
 $\omega\wedge\omega=0$ is equivalent to orthogonality of the electric and magnetic fields (\cite[Section 5.3] {FLS} or \cite{Tartar2}). Noting that
then $d(\alpha\wedge \omega)=\omega\wedge \omega$, a simple argument using integration by parts and Sobolev embedding shows that $\omega \wedge \omega$ is weakly continous in appropriate function spaces. 
 Since in the MHD system we have $E=B \times u$, the pointwise identity
\begin{equation}\label{EB0}
	B\cdot E=0
\end{equation}
must be satisfied in the relaxation of MHD (the set of
weak limits of solutions to MHD). 

A central aspect of this paper is to understand what happens with the compensated compactness quantity $B \cdot E$ below the integrability threshold where $B \cdot E$ is weakly continuous. The recent constructions of irregular Jacobians (determinants of gradient maps and a prototype of differential forms) are particularly relevant for us, see e.g.~\cite{AA,H11,LM16,FMO18}.

Note that \eqref{EB0} is a consequence of Ohm's law \eqref{e:MHD3} but not vice versa -- indeed, a key observation of our analysis is that \eqref{EB0} can be thought of as a suitable relaxation of \eqref{e:MHD3}, sufficiently strong to retain conservation of magnetic helicity but consistent with scenarios of magnetic reconnection and Woltjer-Taylor relaxation (compare again with \cite{Eyi3}).

\bigskip

In the context of weak solutions, it is easy to see using Sobolev embedding that magnetic helicity $\mathcal{H}(B)$ is well-defined provided $B\in L^{3/2}(\T^3)$. Moreover, recall that whenever $3/2 \leq p < \infty$,  solutions $(E,B)$ of the Faraday-Maxwell system \eqref{e:MHD2} with $B \cdot E = 0$ satisfy
$$(B,E) \in L^p \times L^{p'}(\T^3 \times [0,T]) \quad \Longrightarrow \quad \mathcal{H}(B)(t)=\mathcal{H}(B)(0) \textrm{ a.e. } t > 0$$
\cite[Theorem 2.2]{FLS}. Our first result, Theorem \ref{t:Faraday} below, in this paper shows the sharpness of this statement. Combining this result with the techniques introduced in \cite{FLS}, we are able to 'lift' solutions of the Faraday-Maxwell system to the full MHD system in two ways: 
\begin{enumerate}
    \item[(i)] First, in the context of \emph{bounded} weak solutions we show the existence of weak solutions with arbitrary (constant in time) magnetic helicity and at the same time arbitrary time development of energy and cross-helicity - see Corollary \ref{c:bounded};
    \item[(ii)] Secondly, we show the sharpness of the criteria for magnetic helicity conservation by Kang-Lee \cite{KL}; namely, the existence of weak solutions uniformly-in-time in the spatial Lorentz space $L^{3,\infty}$ which do not conserve magnetic helicity - see Corollary \ref{c:L3weak}.
\end{enumerate}

\subsection{Main results}

In order to be able to precisely state our main results, we fix some basic terminology and notation that will be used throughout the paper. 

\begin{itemize}
\item {\bf Spatial domains:} $\Omega\subset\T^3$ or $\Omega\subset\R^3$ denotes a bounded open spatial subset whose boundary $\partial\Omega$ has zero Lebesgue measure;
\item {\bf Space-time domains: }$Q\subset \T^3\times \R$ or $Q\subset \R^3\times\R$ denotes a bounded open subset of space-time whose space-time boundary $\partial Q$ has zero 4D Lebesgue measure;
\item {\bf Time slices: } Given a space-time domain $Q$, for any fixed time $t$ the set $Q(t)$ denotes the time-slice of $Q$, i.e.~$Q(t)=\{x:\,(x,t)\in Q\}$. 
\item {\bf Lebesgue measure: }We will frequently use the notation $|A|$ to denote Lebesgue measure of the set $A$ of appropriate dimension. Thus, $|Q|$ and $|\partial Q|$ refer to 4D Lebesgue measure, whereas $|\Omega|$, $|Q(t)|$ and $|\partial\Omega|$ to 3D Lebesgue measure.
\item {\bf Piecewise constant fields: }A vector field $v:Q\to \R^3$ is said to be piecewise constant if there exists a countable family of pairwise disjoint open subdomains $\{Q_i\}_i$ such that $|Q\setminus \bigcup_iQ_i|=0$ and $v|_{Q_i}$ is constant for each $i$.
\item {\bf Function spaces: }The usual Lebesgue spaces will be denoted by $L^p(\Omega)$ or $L^p(Q)$, respecting the convention above that $\Omega$ is a spatial domain and $Q$ a space-time domain. Appropriate Lebesgue spaces of vector fields will be denoted by $L^p(\Omega;\R^3)$ or $L^p(Q;\R^3)$. The weak $L^p$ spaces of Marcinkiewicz will be referred to in the Lorentz notation $L^{p,\infty}(\Omega)$.
\end{itemize}

Our first main result shows that below the critical integrability, we can restore condition \eqref{EB0} for arbitrary piecewise constant solutions to the Faraday-Maxwell system without being forced to have constant magnetic helicity.

\begin{theorem}\label{t:Faraday}
Let $(\overline{B},\overline{E})\in L^{\infty}(\T^3\times[0,T])$ be a pair of piecewise constant vector fields solving the Faraday-Maxwell system \eqref{e:MHD2} in the sense of distributions, and let $p,p'$ H\"older-dual exponents with $3/2<p<\infty$. Then there exist piecewise constant vector fields $B,E\in L^1(\T^3\times[0,T])$ solving \eqref{e:MHD2} with
$$
B\in L^\infty(0,T;L^{p,\infty}(\T^3)),\quad E\in L^\infty(0,T;L^{p',\infty}(\T^3))
$$  
such that $B\cdot E=0$ for a.e. $(x,t)$ and 
$$
\mathcal{H}(B)(t)=\mathcal{H}(\overline{B})(t)\textrm{ for a.e. }t.
$$
\end{theorem}
The proof, presented in Section \ref{s:Weak solutions of the Faraday-Maxwell system}, relies on an anisotropic version of convex integration in $L^p$ through staircase laminates from 
\cite{F03, AFS} which might be of independent interest.
 
Our second main theorem states that orthogonal solutions of the Faraday-Maxwell system can be ``lifted'' to weak solutions of the full ideal MHD system \eqref{e:MHD} whilst preserving integrability. 

\begin{theorem}\label{t:main}
There exists a geometric constant $M_0>0$ with the following property.

Let $\overline{B}\in L^\infty(0,T;L^2(\T^3)))$ and $\overline{E}\in L^\infty(0,T;L^1(\T^3)))$ be a pair of piecewise constant vector fields solving the Faraday-Maxwell system \eqref{e:MHD2} in the sense of distributions and such that $\overline{E}\cdot \overline{B}=0$. 

Let $\{Q_i\}$ be a countable family of pairwise disjoint open sets on each of which $\overline{B},\overline{E}$ are constant, and let $\zeta_+,\zeta_-\in L^\infty(0,T;L^1(\T^3))$ such that $\zeta_+,\zeta_-\in C(\overline{Q}_i)$ for each $i$ and 
$$
M_0(|\overline{B}|^2+|\overline{E}|)\leq \min\{\zeta_+^2,\zeta_-^2\}\quad\textrm{ for a.e. }(x,t)\in\T^3\times[0,T].
$$
Then there exists a weak solution $(u,B)\in L^{\infty}([0,T];L^2(\T^3))$ of \eqref{e:MHD} such that  
\begin{subequations}
\begin{equation}\label{e:main1}  
|B+u|=\zeta_+\textrm{ and }|B-u|=\zeta_-\textrm{ a.e. $(x,t)$}
\end{equation}
and 
\begin{equation}\label{e:main2}
	\mathcal{H}(B)(t)=\mathcal{H}(\overline{B})(t)\quad\textrm{ a.e. }t.
\end{equation}
\end{subequations}
\end{theorem}

The proof, presented in Section \ref{s:Proof of Theorem 3}, is an extension of our previous paper \cite{FLS}, where we adapted the convex integration scheme from~\cite{DLS09} to be compatible with the non-linear constraint \eqref{EB0}.

Theorem \ref{t:main} suggests that, at least at the level of merely bounded weak solutions, the dynamics of ideal MHD is determined entirely by the behaviour of the Faraday-Maxwell system together with the condition \eqref{EB0} replacing \eqref{e:MHD3} - thus providing a satisfactory mathematical relaxation of the full ideal MHD system. 

There are two particular consequences of this result: First, we have a natural extension of the main result from \cite{FLS} to initial data with arbitrary (a fortiori constant) magnetic helicity:

\begin{corollary}\label{c:bounded}
There exists a geometric constant $M > 0$ with the following property.

Let $h \in \R$ and suppose $e,w \in C([0,T])$ with $e(t) - |w(t)| > M |h|$ for all $t$. Then there exists a weak solution $(u,B)\in L^{\infty}(\T^3\times [0,T])$ of \eqref{e:MHD} such that  
\begin{equation} \label{e:prescribed energy and cross helicity}
\mathcal{E}(u,B)(t)=e(t),\quad \mathcal{W}(u,B)(t)=w(t)\quad\textrm{ for a.e. t}
\end{equation}
and
\begin{equation}
\mathcal{H}(B)(t)=h\quad\textrm{ for a.e. t}.
\end{equation}
\end{corollary}

\medskip

Secondly, we show sharpness of the Kang-Lee result. 

\begin{corollary}\label{c:L3weak}
There exist weak solutions of ideal MHD with 
$$
u,B\in L^\infty(0,T;L^{3,\infty}(\T^3)),
$$
such that neither magnetic helicity, nor energy or cross-helicity are conserved in time.
\end{corollary}

The proofs of these corollaries are quite simple, and will be presented in Section \ref{s:corollaries}. Corollary \ref{c:L3weak} should be compared with \cite{BBV,FL}. In \cite{FL} it was shown that magnetic helicity is conserved by those solutions of ideal MHD which arise as inviscid, non-resistive weak limits of Leray-Hopf solutions. Opposite to this result, in \cite{BBV} it was shown that general $L^\infty_t L^2_x$ solutions (in fact $L^\infty_t H^\beta_x$ for $0 < \beta \ll 1$) of ideal MHD need not preserve magnetic helicity. Corollary \ref{c:L3weak} solves the flexible part of \cite[Conjecture 11]{BVreview} (in the $L^p$ scale). 


\section{Weak solutions of the Faraday-Maxwell system} \label{s:Weak solutions of the Faraday-Maxwell system}

The purpose of this section is to develop a version of convex integration for the Faraday-Maxwell system 
\begin{equation}\label{e:Faraday}
    \begin{split}
   \partial_t B + \nabla \times E &= 0, \\
   \nabla \cdot B &= 0,      
    \end{split}
\end{equation}
and in particular to prove Theorem \ref{t:Faraday}, that is, construct weak solutions $(B,E)$ with $B\cdot E=0$ a.e., which do not conserve magnetic helicity. Convex integration in $L^p$ in relation with 
integrability issues was introduced in \cite{AFS}, based on the staircase laminates 
from \cite{F03}. Such constructions have turned out be useful in a number of problems
\cite{CFM,F04} particularly to obtain lower bounds for singular integrals \cite{BSV}. As $B \cdot E$ is a compensated compactness quantity, our result is inspired
by construction of gradients of homeomorphisms with vanishing Jacobian determinant. In fact, those were inspired by the construction of very weak solutions to elliptic equations \cite{F04,AFS}. Notice that such constructions can only exist
in function spaces where the corresponding compensated compactness quantity is no longer weakly continuous. Thus, the dichotomy between weak compactness and rigidity versus lack of compactness and flexible convex integration solution arises once more.
   Let us further emphasize that the construction here is anisotropic, and thus is based in a curvy staircase laminate which is a new feature in the literature. 
   The known convex integration constructions applied to such curvy laminates yield $L^{3,\infty}_{x,t}$ solutions. In order to achieve $L^\infty_tL^{3,\infty}_x$, it is needed
to control what happens at almost every  time slice. The
innovations introduced in the paper to deal with this issue are also  of potential use elsewhere.

In modifying piecewise constant vector fields $(B,E):Q\to \R^3\times \R^3$ with $\textrm{div }B=0$ and $\partial_tB+\textrm{curl }E=0$, a typical situation is as follows. Suppose $Q_0\subset Q$ is a subdomain where $(B^{(0)},E^{(0)})=(B_0,E_0)$ are constant. We wish to ``replace'' the constant value $(B_0,E_0)$ by another pair of vector fields $(B,E):Q_0\to\R^3\times \R^3$ such that the ``glued'' vector fields, defined by 
$$
(B^{(1)},E^{(1)})=\begin{cases} (B_0,E_0)&\textrm{ outside }Q_0\\ (B,E)&\textrm{ in }Q_0\end{cases}
$$
still satisfy $\div B^{(1)}=0$ and $\partial_tB^{(1)}+\curl E^{(1)}=0$. It is easy to check that, in general, a necessary and sufficient condition for this is that 
\begin{equation}\label{e:boundarycondition}
\textrm{ The extensions }(\overline{B},\overline{E})=\begin{cases} (B_0,E_0)&\textrm{outside }Q_0\\ (B,E)&\textrm{in }Q_0\end{cases}\textrm{ satisfy \eqref{e:Faraday} in }\mathcal{D}'(\R^4). 
\end{equation}

\subsection{The basic staircase}

We start by constructing a discrete ''staircase" laminate in the plane $\R^2$. Recall that laminates in the plane are defined with respect to separate convexity, i.e.~corresponding to the wave cone $\{(x,y)\in \R^2:\,x=0\textrm{ or }y=0\}$.

\begin{lemma}\label{l:laminate}
Let $1<p<\infty$.
For any $n\in\N$ and $\beta>1$ we define
$$
\mu_n=\sum_{k=0}^{n-1}\left[\lambda_1^{(k)}\delta_{(0,\beta^{k(p-1)})}+\lambda_2^{(k)}\delta_{(\beta^{k+1},0)}\right]+\gamma^{(n)}\delta_{(\beta^n,\beta^{n(p-1)})},
$$
where 
\begin{equation}\label{e:laminateweights}
\lambda_1^{(k)}=(1-\frac{1}{\beta})\gamma_k,\quad \lambda_2^{(k)}=\frac{1}{\beta}(1-\frac{1}{\beta^{p-1}})\gamma_k,\quad \gamma^{(k)}=\beta^{-kp}.
\end{equation}
Then $\mu_n$ is a laminate on $\R^2$ with barycenter $(1,1)$. 
\end{lemma}

\begin{proof}
The proof is by induction on $n$. The case $n=0$ is trivial, since $\mu_0=\delta_{(1,1)}$. The inductive step $n\mapsto n+1$ proceeds by the following splitting procedure:
\begin{equation}\label{e:basicsplitting}
\begin{split}
\delta_{(\beta^n,\beta^{n(p-1)})} \mapsto& (1-\tfrac{1}{\beta})\delta_{(0,\beta^{n(p-1)})}+\tfrac{1}{\beta}\delta_{(\beta^{n+1},\beta^{n(p-1)})}\\
\mapsto& (1-\tfrac{1}{\beta})\delta_{(0,\beta^{n(p-1)})}+\tfrac{1}{\beta}(1-\tfrac{1}{\beta^{p-1}})\delta_{(\beta^{n+1},0)}+\tfrac{1}{\beta^p}\delta_{(\beta^{n+1},\beta^{(n+1)(p-1)})}.
\end{split}
\end{equation}
Thus, we obtain $\gamma^{(n+1)}=\tfrac{1}{\beta^p}\gamma^{(n)}$. Combined with $\gamma_1=1$ we obtain $\gamma^{(n)}=\beta^{-np}$. The expressions for $\lambda^{(k)}_1$ and $\lambda_2^{(k)}$ are analogous.
\end{proof}

Now given vectors $B_0,E_0$ we can embed the two dimensional laminate  in the plane
spanned by them.

\begin{corollary}\label{c:laminate}
For any $1<p<\infty$, any $\beta>1$, any $(B_0,E_0)\in\R^3\times\R^3$ and any $n\in\N$ the probability measure
$$
\sum_{k=0}^{n-1}\left[\lambda_1^{(k)}\delta_{(0,\beta^{k(p-1)}E_0)}+\lambda_2^{(k)}\delta_{(\beta^{k+1}B_0,0)}\right]+\gamma^{(n)}\delta_{(\beta^nB_0,\beta^{n(p-1)}E_0)},
$$
with $\lambda_1^{(k)}$, $\lambda_2^{(k)}$, $\gamma^{(k)}$ defined as in \eqref{e:laminateweights}, is a laminate on $\R^3\times\R^3$ with respect to the wave cone $\Lambda=\{(B,E):B\cdot E=0\}$, with barycenter $(B_0,E_0)$.
\end{corollary}

\subsection{Basic construction}

Here we recall and appropriately adapt the basic so called ``roof-construction'' for localized plane-waves, see e.g \cite{Kir}. in the following we denote by $\textrm{Lip}_0(Q)$ the set of Lipschitz functions on $\overline{Q}$ such that $f=0$ on $\partial Q$. 

\begin{lemma}\label{l:basicBE}
Let $B_1,E_1,B_2,E_2\in \R^3$ with $(B_2-B_1)\cdot(E_2-E_1)=0$ and $\lambda_1,\lambda_2\in(0,1)$ with $\lambda_1+\lambda_2=1$. For any open bounded domain $Q\subset\R^4$ with $|\partial Q|=0$ and any $r,\eps>0$ there exist piecewise constant vector fields $B,E\in L^{\infty}(Q;\R^3)$ satisfying \eqref{e:Faraday} and the boundary conditions given by 
$$
(B_0,E_0)=\lambda_1(B_1,E_1)+\lambda_2(B_2,E_2)
$$ 
in the sense of \eqref{e:boundarycondition}, with the following properties:
\begin{subequations}\label{e:basicBEprops}
\begin{itemize}
\item $Q$ admits a pairwise disjoint decomposition 
\begin{equation}
Q=Q^{(1)}\cup Q^{(2)}\cup Q^{(error)}\cup \mathcal{N}
\end{equation}
where $\mathcal{N}$ a nullset, $Q^{(1)}, Q^{(2)}$ and $Q^{(error)}$ are open sets where $(B,E)$ is locally constant, and such that 
$(B,E)=(B_i,E_i)$ in $Q_i$, $i=1,2$ and $|B-B_0|+|E-E_0|<r$ in $Q^{(error)}$. 
\item For $i=1,2$ and any $t\in\R$
\begin{equation}\label{e:basic-estt}
|Q^{(error)}(t)|+\frac{1}{\lambda_i}|Q^{(i)}(t)|\leq (1+\eps)|Q(t)|
\end{equation}
and moreover
\begin{equation}\label{e:basic-est}
|Q^{(error)}|\leq \eps|Q|.
\end{equation}
\item There exists a vector potential $\tilde A\in Lip_0(Q)$ with 
\begin{equation}\label{e:basic-potential}
B_0+\curl \tilde A=B\textrm{ and }|\tilde A|\leq \eps.
\end{equation} 
\end{itemize}
\end{subequations}
\end{lemma}

\begin{proof}
Let $\bar{B}=B_2-B_1$ and $\bar{E}=E_2-E_1$. 
Since $\bar{B}\cdot\bar{E}=0$, there exist vectors $\xi,\eta \in \R^3$ with $|\xi|=|\eta|=1$ such that
$$
\bar{B}=|\bar{B}|\xi\times\eta,\quad \bar{E}=|\bar{E}|\xi.
$$
Let us now consider first a polyhedral spatial domain $\Omega$ and space-time domain of the form 
\begin{equation}
Q=\Omega\times (t_0,t_1).
\end{equation}
Define
\begin{align*}
E(x,t)&=E_0+\tilde E=E_0+|\bar{E}|\nabla f_N(x,t) - |\bar{B}| \partial_t [\eta f(x,t)],\\
B(x,t)&=B_0+\tilde B=B_0+|\bar{B}|\textrm{curl }(\eta f_N(x,t))=B_0+|\bar{B}|\nabla f(x,t) \times\eta,
\end{align*}
where
\begin{align*}
f_N(x,t)&=\min\left\{r\dist(x,\partial\Omega),f_N^p(x,t)\right\}\\
f_N^p(x,t)&=\min\left\{r(t-t_0)_+,r(t_1-t)_+,\frac{1}{N}h(Nx\cdot\xi)\right\},
\end{align*}
and $h:\R\to\R$ is a 1-periodic non-negative Lipschitz function with $h'(s)\in \{-\lambda_2,\lambda_1\}$ for a.e.~$s\in\R$.
Observe that for every $t$ the function $x\mapsto f^p_N(x,t)$ is a periodic piecewise affine Lipschitz function such that
\begin{equation}
\nabla f_N^p(x,t)\in \left\{0,-\lambda_2\xi,\lambda_1\xi\right\}\quad \textrm{ a.e. }x
\end{equation}
with respective volume fractions $\mu(t), \lambda_1(1-\mu(t)), \lambda_2(1-\mu(t))$, relative to one period. 

Since $\Omega$ is a polygonal domain, $f_N$ is piecewise affine. Thus, by definition, there exists an open subset $\mathring{Q}\subset Q$ such that $|Q\setminus \mathring{Q}|=0$ and $f_N$ is locally affine (as well as $\nabla f_N$ is locally constant) in $\mathring{Q}$. 
Let us define $\mathcal{N}=Q\setminus \mathring{Q}$, 
\begin{align*}
    Q^{(1)}&=\left\{(x,t)\in \mathring{Q}:\,\nabla f_N(x,t)=-\lambda_2\xi\right\},\\
    Q^{(2)}&=\left\{(x,t)\in \mathring{Q}:\,\nabla f_N(x,t)=\lambda_1\xi\right\},\\
    Q^{(error)}&=\left\{(x,t)\in \mathring{Q}:\,\nabla f_N(x,t)\notin \{-\lambda_2\xi,\lambda_1\xi\}\right\}.
\end{align*}
It then follows that for all $t$
\begin{align*}
\left|Q^{(1)}(t)\right|&\leq \lambda_1(1-\mu(t))|Q(t)|,\\
\left|Q^{(2)}(t)\right|&\leq \lambda_2(1-\mu(t))|Q(t)|,\\
\left|Q^{(error)}(t)\right|&\leq (\mu(t)+O(\tfrac{1}{N}))|Q(t)|,
\end{align*} 
the latter following from the pointwise bound $|f_N^p|\leq \frac{1}{N}$. Then, after eliminating $\mu(t)$ and  from the above relationships and choosing $N$ sufficiently large in terms of $\eps$, we obtain \eqref{e:basic-estt}. To deduce \eqref{e:basic-est} we may again use the pointwise bound on $f_N^p$ to see that in fact 
\begin{equation*}
    Q^{(error)}\subset \left\{(x,t)\in Q:\,r\dist\left((x,t),\partial Q\right)\leq \frac{1}{N}\right\},
\end{equation*}
so that \eqref{e:basic-est} follows by choosing $N$ sufficiently large. Finally, observe that $\tilde A=|\bar B|\eta f_N$, so that \eqref{e:basic-potential} also follows from choosing $N$ sufficiently large. This concludes the proof for the case of a cylindrical polyhedral set $Q=\Omega\times (t_0,t_1)$.

\smallskip

For a general open space-time domain $Q$ we find a pairwise disjoint countable family of cylindrical polyhedral sets $Q_k\subset Q$ such that $|Q\setminus \bigcup_kQ_k|=0$, apply the above to obtain 
$(\tilde B_k,\tilde E_k)$ in $Q_k$ with estimates \eqref{e:basic-estt}-\eqref{e:basic-potential}. Defining
\begin{equation*}
    Q^{(1)}=\bigcup_kQ^{(1)}_k,\, Q^{(2)}=\bigcup_kQ^{(2)}_k,\,Q^{(error)}=\bigcup_kQ^{(error)}_k
\end{equation*}
as well as $\mathcal{N}=(Q\setminus \bigcup_kQ_k)\cup\bigcup_k\mathcal{N}_k$ leads to the required properties.
\end{proof}

\begin{remark} \label{r:magnetichelicitycomputation}
We compute explicitly the change of magnetic helicity in Lemma \ref{l:basicBE}. In the proof above, let us denote by $f$ the piecewise affine Lipschitz function which vanishes outside the polyhedral sets $Q_k$ and is of the form $f = f_N$ in each $Q_k$, so that $\tilde{A} = |\bar{B}| f \eta$. We denote $Q(t) \defeq \{x \in \R^3: (x,t) \in Q\}$ for $t \in \R$. By integrating by parts and using the facts that $\tilde{A}|_{\partial Q} = 0$ and $\tilde{A} \cdot \tilde{B} = 0$ we get
\begin{align*}
\int_{\R^3}[(A_0 + \tilde{A}) \cdot B - A_0 \cdot B_0]
&= \int_{Q(t)} \tilde{A} \cdot B_0 + \int_{Q(t)} A_0 \cdot \tilde{B}
= 2 \int_{Q(t)} \tilde{A} \cdot B_0 \\
&= 2 |\bar{B}| \eta \cdot B_0 \int_{Q(t)} f.
\end{align*}
\end{remark}

\bigskip

Next, we intend to implement the basic splitting \eqref{e:basicsplitting} in the construction of the staircase laminate in Corollary \ref{c:laminate}. To this end we fix vectors $B_0,E_0$, $\beta>1$ and set
$$
B_n=\beta^nB_0,\quad E_n=\beta^{(p-1)n}E_0.
$$

\begin{lemma}[Approximation of Steps]\label{l:step}
For any $n\in \N$, any open bounded domain $Q\subset\R^4$ with $|\partial Q|=0$ and any $r,\eps>0$ there exist piecewise constant vector fields $B,E\in L^{\infty}(Q;\R^3)$ satisfying \eqref{e:Faraday} and the boundary conditions given by 
$(B_n,E_n)$ in the sense of \eqref{e:boundarycondition}, with the following properties:
\begin{subequations}\label{e:stepprops}
\begin{itemize}
\item $Q$ admits a pairwise disjoint decomposition 
\begin{equation}
Q=Q^{(good)}\cup Q^{(inductive)}\cup Q^{(error)}\cup \mathcal{N}
\end{equation}
where $\mathcal{N}$ a nullset, $Q^{(good)}, Q^{(inductive)}$ and $Q^{(error)}$ are open sets where $(B,E)$ is locally constant with 
\begin{equation}
\begin{split}
|B||E|&=0\textrm{ in }Q^{(good)}, \\
(B,E)&=(B_{n+1},E_{n+1})\textrm{ in }Q^{(inductive)},\\
\dist(B,\{B_n,B_{n+1}\})&+|E-E_n|<r\textrm{ in }Q^{(error)}.
\end{split}
\end{equation}
\item For all $t$ we have
\begin{equation}
|Q^{(error)}(t)|+\beta^p|Q^{(inductive)}(t)|\leq (1+\eps)|Q(t)|\label{e:proportionalspace-basic}
\end{equation}
and
\begin{equation}\label{e:smallspacetime-basic}
|Q^{(error)}|< \eps.
\end{equation}
\item There exists a vector potential $\tilde A\in Lip_0(Q)$ such that, for any vector potential $A_0$ of $B_0$, 
\begin{equation}
B_0+\curl \tilde A=B, \; |\tilde A|\leq \eps \textrm{ and } \int_{\R^3} [(A_0+\tilde{A}) \cdot B - A_0 \cdot B_0] \, dx = 0 \textrm{ a.e. } t \in \R.
\end{equation}
\end{itemize}
\end{subequations}
\end{lemma}

\begin{proof}
We may assume without loss of generality that $\eps<1$.
In the first step we apply Lemma \ref{l:basicBE} with the elementary splitting
\begin{equation} \label{e:splitting}
\delta_{(\beta^{n}B_0,\beta^{n(p-1)}E_0)}\mapsto (1-\tfrac{1}{\beta})\delta_{(0,\beta^{n(p-1)}E_0)}+\tfrac{1}{\beta}\delta_{(\beta^{n+1}B_0,\beta^{n(p-1)}E_0)}.
\end{equation}
We obtain $(B^{(1)},E^{(1)})$ and the decomposition
$$
Q=Q^{(1)}\cup Q^{(2)}\cup Q^{(error,1)}\cup \mathcal{N}^{(1)},
$$
where $\mathcal{N}^{(1)}$ is a null-set, 
$$
(B^{(1)},E^{(1)})=\begin{cases}
(0,\beta^{n(p-1)}E_0)&\textrm{ in }Q^{(1)},\\
(\beta^{n+1}B_0,\beta^{n(p-1)}E_0)&\textrm{ in }Q^{(2)}\\
\end{cases}
$$
and for all $t$
\begin{align*}
\left|Q^{(error,1)}(t)\right|+\beta\left|Q^{(2)}(t)\right|&\leq (1+\eps/4)|Q(t)|,\\
\left|Q^{(error,1)}\right|&\leq \eps/2|Q|.
\end{align*}
Furthermore, $B^{(1)}=B_0+\curl A^{(1)}$, with $|A^{(1)}|\leq \eps/2$. We then use Remark \ref{r:magnetichelicitycomputation} to compute the change of magnetic helicity. In the elementary splitting \eqref{e:splitting}, $(\beta^{n+1} B_0,0) = (\bar{B},\bar{E}) = |\bar{B}| \xi \times \eta, 0)$ for some $\xi,\eta \in \R^3$ with $|\xi| = |\eta| = 1$, and so Remark \ref{r:magnetichelicitycomputation} gives
\begin{equation} \label{e:changeofmagnetichelicity}
\int_{\R^3} [(A_0 + A^{(1)}) \cdot B^{(1)} - A_0 \cdot B_0] \, dx = 0.
\end{equation}

Then we apply Lemma \ref{l:basicBE} in $Q^{(2)}$ with the second elementary splitting
\begin{equation} \label{e:splitting2}
\delta_{(\beta^{n+1}B_0,\beta^{n(p-1)}E_0)}\mapsto (1-\tfrac{1}{\beta^{p-1}})\delta_{(\beta^{n+1}B_0,0)}+\tfrac{1}{\beta^{p-1}}\delta_{(\beta^{n+1}B_0,\beta^{(n+1)(p-1)}E_0)}.
\end{equation}
We obtain $(B,E)$ and the decomposition
$$
Q^{(2)}=Q^{(3)}\cup Q^{(4)}\cup Q^{(error,2)}\cup \mathcal{N}^{(2)},
$$
where $\mathcal{N}^{(2)}$ is a null-set, 
$$
(B,E)=\begin{cases}
(\beta^{n+1}B_0,0)&\textrm{ in }Q^{(3)},\\
(\beta^{n+1}B_0,\beta^{(n+1)(p-1)}E_0)&\textrm{ in }Q^{(4)}\\
\end{cases}
$$
and for all $t$
\begin{align*}
\left|Q^{(error,2)}(t)\right|+\beta^{p-1}\left|Q^{(4)}(t)\right|&\leq (1+\eps/4)|Q^{(2)}(t)|,\\
\left|Q^{(error,2)}\right|&\leq \eps/2|Q^{(2)}(t)|.
\end{align*}
Furthermore, $B=B^{(1)}+\curl A^{(2)}$, with $|A^{(2)}|\leq \eps/2$.
Set
\begin{equation}
Q^{(inductive)}:=Q^{(4)},\quad Q^{(error)}=Q^{(error,1)}\cup Q^{(error,2)}.
\end{equation}
Then for every $t$
\begin{align*}
\left|Q^{(error)}(t)\right|+\beta^p\left|Q^{(inductive)}(t)\right|&\leq \left|Q^{(error,1)}(t)\right|+\beta\left|Q^{(error,2)}(t)\right|+\beta^p\left|Q^{(4)}(t)\right|\\
&\leq \left|Q^{(error,1)}(t)\right|+\beta(1+\eps/4)\left|Q^{(2)}(t)\right|\\
&\leq (1+\eps/4)^2\left|Q(t)\right|\\
&\leq (1+\eps)\left|Q(t)\right|
\end{align*}
and 
\begin{equation}
   \left|Q^{(error)}\right|\leq \left|Q^{(error,1)}\right|+\left|Q^{(error,2)}\right|\leq \eps|Q| 
\end{equation}
as required. Moreover, $B=B_0+\curl(A^{(1)}+A^{(2)})$, with $|A^{(1)}+A^{(2)}|\leq \eps$. In the elementary splitting \eqref{e:splitting2} we can write $(0,\beta^{(n+1)(p-1)} E_0) = (\bar{B},\bar{E}) = (|\bar{B}| \xi \times \eta, |\bar{E}| \xi)$ for $\xi = \bar{E}/|\bar{E}|$ and any $\eta \in \R^3$ with $|\eta|=1$. By letting $\eta \cdot B_0 = 0$, Remark \ref{r:magnetichelicitycomputation} and \eqref{e:changeofmagnetichelicity} give
\[\int_{\R^3} [(A_0 + A^{(1)} + A^{2}) \cdot B - A_0 \cdot B_0] \, dx = 0 \quad \textrm{a.e. } t \in \R.\]
This concludes the proof. 
\end{proof}

\subsection{The staircase construction}

The basic construction above has the following structure: up-to an ``error set'' $Q^{(error)}$ the distribution of values of the pair of vector fields $(B,E)$ agrees with the probability measure (laminate) arising in \eqref{e:basicsplitting}. There are two types of control on the size of the error set: small space-time measure \eqref{e:smallspacetime-basic} on the one hand, and on the other hand control uniformly in time by the proportion of mass moved to the inductive set \eqref{e:proportionalspace-basic}. 

In the following we will iterate the basic construction, to inductively ``push'' the mass in $Q^{(inductive)}$ to infinity. The balance between the error created at each step and the amount mass pushed inductively further will be quantified by estimate \eqref{e:Qerror1}.

\begin{proposition}\label{p:staircase}
Let $B_0,E_0\in \R^3$, $1<p<\infty$ and $Q\subset\R^4$ an open bounded domain with $|\partial Q|=0$. For any $\beta>1$ and $\eps>0$ there exist piecewise constant vector fields $B,E\in L^{\infty}(Q;\R^3)$ satisfying \eqref{e:Faraday} and the boundary conditions given by 
$(B_0,E_0)$ in the sense of \eqref{e:boundarycondition}, with the following properties: 
\begin{subequations}\label{e:staircaseprops}
\begin{itemize}
\item $Q$ admits a decomposition 
\begin{equation}\label{e:gooderror}
Q=Q^{(good)}\cup Q^{(error)}\cup \mathcal{N}
\end{equation}
where $\mathcal{N}$ is a nullset, $(B,E)$ is locally constant in the open sets $Q^{(good)}$ and $Q^{(error)}$ and such that
$|B||E|=0$ in $Q^{(good)}$.
\item For all $t$ and all $s>1$ we have
\begin{equation}\label{e:Qerror1}
\mathcal{I}(s,t)\leq \beta^{2(p+1)}|Q(t)|\min(|B_0|^p+|E_0|^{p'},s), 
\end{equation}
where
\begin{equation*}
\mathcal{I}(s,t):=\int_{Q^{(error)}(t)}\min\{|B|^p+|E|^{p'},s\}\,dx+s\left|\{x\in Q^{(good)}(t):|B|^p+|E|^{p'}>s\}\right|
\end{equation*}
and $p'$ is the H\"older dual of $p$.
\item Furthermore, 
\begin{equation}\label{e:Qerror2}
\int\int_{Q^{(error)}}|B|^p+|E|^{p'}\,dxdt\leq \eps.
\end{equation}
\item There exists a vector potential $\tilde A\in Lip_0(Q)$ with 
\begin{equation}\label{e:Qerror3}
B_0+\curl \tilde A=B, \; |\tilde A|\leq \eps \textrm{ and } \int_{\R^3} [(A_0+\tilde{A}) \cdot B - A_0 \cdot B_0] \, dx = 0 \textrm{ a.e. } t \in \R.
\end{equation} 
\end{itemize}
\end{subequations}
\end{proposition}

\begin{proof}
Based on Lemma \ref{l:step} we define inductively a sequence $B^{(n)},E^{(n)}\in L^{\infty}(Q;\R^3)$ satisfying \eqref{e:Faraday} and the boundary conditions given by $(B_0,E_0)$ in the sense of \eqref{e:boundarycondition}, with the following properties: Firstly, 
we have the pairwise disjoint decomposition
\begin{equation}
Q=\bigcup_{k=0}^{n-1}(Q_{k+1}^{(error)}\cup Q_{k+1}^{(good)})\cup Q_n^{(inductive)} \cup \mathcal{N}_n,
\end{equation}
where $\mathcal{N}_n$ is a null-set, $Q_k^{(good)}$, $Q_n^{(inductive)}$ and $Q_k^{(error)}$ are open sets where $(B_n,E_n)$ is locally constant, such that
\begin{align*}
|B^{(n)}||E^{(n)}|&=0\textrm{ in }Q_n^{(good)}\\
(B^{(n)},E^{(n)})&=(\beta^{n}B_0,\beta^{n(p-1)}E_0)\textrm{ in }Q_n^{(inductive)}.
\end{align*}
Secondly, $B^{(n+1)}=B^{(n)}+\curl A^{(n)}$ with $|A^{(n)}|\leq \eps 2^{-n-1}$.

\smallskip

We start by defining 
$$
(B^{(1)},E^{(1)}):\equiv (B_0,E_0),\quad Q^{(inductive)}_0:=Q,\quad Q^{(good)}_0=Q^{(error)}_0=\mathcal{N}_0=\emptyset.
$$
To obtain $(B^{(n+1)},E^{(n+1)})$ we apply Lemma \ref{l:step} to $(B^{(n)},E^{(n)})$ in $Q^{(inductive)}_n$ with small parameters $r_n,\eps_n>0$ chosen below, with $\eps_n<\eps 2^{-n-1}$. Then we obtain
\begin{equation*}
    Q^{(inductive)}_n=Q^{(good)}_{n+1}\cup Q^{(error)}_{n+1}\cup Q^{(inductive)}_{n+1}\cup \mathcal{N}'_{n+1}
\end{equation*}
with
$$
(B^{(n+1)},E^{(n+1)})\approx_{r_k}(\beta^{k}B_0,\beta^{k(p-1)}E_0)\textrm{ or }(\beta^{k+1}B_0,\beta^{k(p-1)}E_0)\textrm{ in }Q_{k+1}^{(error)}
$$
for all $k<n+1$, where in the last line $\approx_{r_k}$ means that the norm of the difference is bounded by $r_k$.
In particular we may ensure by the choice of $r_k$ that 
$$
(|B_0|^p+|E_0|^{p'})\beta^{kp-1}\leq |B|^p+|E|^{p'}\leq (|B_0|^p\beta^p+|E_0|^{p'}\beta)\beta^{kp}
\textrm{ in }Q^{(error)}_{k+1}.
$$
Then we have for all $t$
\begin{equation}\label{e:Qn-1}
\left|Q_{n+1}^{(error)}(t)\right|+\beta^p\left|Q_{n+1}^{(inductive)}(t)\right|
\leq \left|Q_n^{(inductive)}(t)\right|+\eps_n|Q(t)|,
\end{equation}
and furthermore
$$
\left|Q_{n+1}^{(error)}\right|\leq \eps_{n}.
$$
The parameters $\eps_n>0$, $n=0,1,2,\dots$ will be chosen below, for the moment let us merely specify that they satisfy
\begin{equation}\label{e:epsn-sum}
\sum_{n=0}^\infty \beta^{np}\eps_n<\beta-1. 
\end{equation}
Such condition and \eqref{e:Qn-1}  immediately imply that, for all $t$,
\begin{equation}\label{e:inductivet}
|Q_n^{(inductive)}(t)|\leq \beta^{-np}|Q(t)|.
\end{equation}
Also, by construction, for any $n\geq k$ 
\begin{equation}\label{e:inductionkn}
(B^{(n)},E^{(n)})=(B^{(k)},E^{(k)})\,\textrm{ outside }Q^{(inductive)}_k.
\end{equation}
Since the measure of $Q^{(inductive)}_k$ tends to $0$, it follows that the sequence $(B^{(n)},E^{(n)})$ converges to a limit $(B,E)$ for almost every $(x,t)$. In the following we derive properties of this limit. To start with we observe that, declaring 
\begin{equation*}
    Q^{(good)}=\bigcup_{k=0}^\infty Q^{(good)}_k,\quad Q^{(error)}=\bigcup_{k=0}^\infty Q^{(error)}_k 
\end{equation*}
the decomposition \eqref{e:gooderror} holds.

\smallskip

We turn to estimate \eqref{e:Qerror1}. From \eqref{e:Qn-1} we 
obtain, using \eqref{e:epsn-sum},
\begin{equation}\label{e:Qkerrorkp}
\begin{split}
\beta^{np}\left|Q_{n}^{(inductive)}(t)\right|+\sum_{k=0}^{n-1}\beta^{kp}\left|Q_{k+1}^{(error)}(t)\right|&\leq \left|Q(t)\right|+\sum_{k=0}^n\beta^{kp}\eps_{k}|Q(t)|\\
&\leq \beta|Q(t)|.
\end{split}
\end{equation}
Now let $n\in \N$ and let 
\begin{equation}\label{e:specials}
   s_n=(|B_0|^p\beta^p+|E_0|^{p'}\beta)\beta^{np}. 
\end{equation}
Then 
\begin{align*}
    \left\{(x,t)\in Q^{(error)}:\, |B|^p+|E|^{p'}<s_n\right\}&=\bigcup_{k=0}^{n-1}Q_{k+1}^{(error)}\,,\\
    \left\{(x,t)\in Q:\,|B|^p+|E|^{p'}>s_n\right\}&=Q_n^{(inductive)}\,.
\end{align*}
Therefore, for every $t$ 
\begin{align*}
    \mathcal{I}(s_n,t)&=\int_{\{x\in Q^{(error)}(t):|B|^p+|E|^{p'}\leq s_n\}}|B|^p+|E|^{p'}\,dx+s_n\left|\{x\in Q(t):|B|^p+|E|^{p'}>s_n\}\right|\\
    &\leq \sum_{k=0}^{n-1}\int_{Q_{k+1}^{(error)}(t)}|B|^p+|E|^{p'}\,dx+s_n\left|Q^{(inductive)(t)}_n\right|\\
    &\leq (|B_0|^p\beta^p+|E_0|^{p'}\beta)\left(\sum_{k=0}^{n-1}\beta^{kp}|Q_{k+1}^{(error)}|+\beta^{np}|Q_n^{(inductive)}(t)| \right)\\
    &\leq (|B_0|^p\beta^p+|E_0|^{p'}\beta)\beta|Q(t)|\\
    &\leq \beta^{2+p}(|B_0|^p+|E_0|^{p'})|Q(t)|.
\end{align*}
where we have used the definition of $s_n$, \eqref{e:Qkerrorkp} and that on ${Q_{k+1}^{(error)}(t)}$ $|B|^p+|E|^{p'}$ is bounded by $s_k$.

Next, let $s>s'$. From the elementary inequality
$\min(|B|^p+|E|^{p'},s)\leq \tfrac{s}{s'}\min(|B|^p+|E|^{p'},s')$ we easily deduce that $I(s,t)\leq \tfrac{s}{s'}I(s',t)$ for every $t$. Now, for any $s\geq s_0$ (the latter defined in \eqref{e:specials} with $n=0$) there exists $n\in\N$ such that $s_n\leq s\leq s_n\beta^p$. Consequently 
$$
I(s,t)\leq \frac{s}{s_n}I(s_n,t)\leq \beta^{2(p+1)}(|B_0|^{p}+|E_0|^{p'})|Q(t)|,
$$
as claimed in \eqref{e:Qerror1}. On the other hand for $1\leq s\leq s_0$ we may use the trivial estimate $I(s,t)\leq s|Q(t)|$, from which 
\eqref{e:Qerror1} also follows.

\smallskip

The estimate \eqref{e:Qerror3} follows from
\begin{align*}
\int\int_{Q^{(error)}}|B|^p+|E|^{p'}\,dxdt &\leq (|B_0|^p\beta^p+|E_0|^{p'}\beta)\sum_{k=0}^\infty|Q_{k+1}^{(error)}|\beta^{kp}\\
&\leq (|B_0|^p\beta^p+|E_0|^{p'}\beta)\sum_{k=0}^\infty\eps_k\beta^{kp},
\end{align*}
and an appropriate choice of $\eps_n>0$, $n=0,1,2,\dots$. 

\smallskip

Finally, observe that \eqref{e:Qerror1} implies uniform-in-time weak $L^p-L^{p'}$ bounds for $(B,E)$, which also clearly hold for the sequence $(B^{(n)},E^{(n)})$ and $p,p'>1$. We deduce that in fact $(B^{(n)},E^{(n)})\to (B,E)$ strongly in $L^1(Q)$. This in turn implies that $\div B=0$, $\partial_tB+\curl E=0$ in $Q$, and the required boundary conditions hold in the sense of \eqref{e:boundarycondition}. This concludes the proof.


\end{proof}

\subsection{Iterating the staircase construction}

In this section we iterate the staircase construction in order to successively remove the error in $\Omega^{error}$. 

\begin{theorem}\label{t:BE}
Let $B_0,E_0\in \R^3$, $1<p<\infty$ and $Q\subset\R^4$ an open bounded domain with $|\partial Q|=0$. For any $\eps>0$ there exist piecewise constant vector fields $B,E\in L^{\infty}(Q;\R^3)$ satisfying \eqref{e:Faraday} and the boundary conditions given by 
$(B_0,E_0)$ in the sense of \eqref{e:boundarycondition} with the following conditions:
\begin{itemize}
    \item $|B||E|=0$ for a.e. $(x,t)\in Q$;
\item For all $t$ and any $s>1$
\begin{equation}\label{e:weakLp}
\left|\left\{x\in Q(t):\,|B|^p+|E|^{p'}>s\right\}\right|\leq \frac{2}{s}|Q(t)|\min(|B_0|^p+|E|_0|^{p'},s),
\end{equation}
so that, in particular, $B\in L^{\infty}_tL^{p,\infty}_x$ and $E\in L^{\infty}_tL^{p',\infty}_x$.
\item There exists a vector potential $\tilde A\in Lip_0(Q)$ with 
\begin{equation}
B_0+\curl \tilde A=B, \; |\tilde A|\leq \eps \textrm{ and } \int_{\R^3} [(A_0+\tilde{A}) \cdot B - A_0 \cdot B_0] \, dx = 0 \textrm{ a.e. } t \in \R.
\end{equation} 
\end{itemize}
\end{theorem}

\begin{proof}
We construct inductively a sequence of piecewise constant vector fields $B_q,E_q$ satisfying \eqref{e:Faraday} and the boundary conditions given by $(B_0,E_0)$ in the sense of \eqref{e:boundarycondition}, and with following properties: there exists a pairwise disjoint decomposition
\begin{equation*}
    Q=Q^{(good)}_q\cup Q^{(error)}_q\cup \mathcal{N}_q
\end{equation*}
where $\mathcal{N}_q$ is a nullset, $Q_q^{(good)}$ and $Q_q^{(error)}$ are open sets where $B_q,E_q$ are locally constant, and $|B_q||E_q|=0$ in $Q^{(good)}$. 
This will be complemented with the inductive estimates
\begin{align}
\int\int_{Q^{error}_{q+1}}1+|B_{q+1}|^p+|E_{q+1}|^{p'}\,dxdt&\leq \frac{1}{2}\int\int_{Q^{error}_{q}}1+|B_q|^p+|E_q|^{p'}\,dxdt\,,\label{e:inductive1}\\
I_{q+1}(s,t)&\leq \beta_{q+1}^{2(p+1)}I_{q}(s,t)\textrm{ for all $t$ and $s\geq 1$},\label{e:inductive2}
\end{align}
where
\begin{equation*}
\mathcal{I}_q(s,t):=\int_{Q_q^{(error)}(t)}\min\{|B_q|^p+|E_q|^{p'},s\}\,dx+s\left|\{x\in Q_q^{(good)}(t):|B_q|^p+|E_q|^{p'}>s\}\right|\,.
\end{equation*}
Furthermore, $B_{q+1}=B_q+\curl A_q$ with $|A_q|\leq \eps 2^{-q-1}$.
We start with the constant maps $(B_0,E_0)$ and set $Q_0^{error}=Q$. 

\medskip

To obtain $(B_{q+1},E_{q+1})$ from $(B_q,E_q)$ we consider the decomposition
$$
Q_q^{(error)}=\bigcup_i Q_{q,i}
$$
with constant values $(B_q,E_q)\equiv (B_q^i,E_q^i)$ on $Q_{q,i}$. In each $Q_{q,i}$ we replace $(B_q^i,E_q^i)$ by the construction from Proposition \ref{p:staircase} with $(B_0,E_0)$ given by $(B_q^i,E_q^i)$ and small parameters $\beta_{q+1}>1,\eps_{q+1}>0$ still to be fixed. 
We obtain for each $i$ a new pairwise disjoint decomposition
\begin{equation*}
    Q_{q,i}=Q_{q,i}^{(good)}\cup Q_{q,i}^{(error)}\cup \mathcal{N}_{q,i}
\end{equation*}
with associated estimates, corresponding to \eqref{e:Qerror1}-\eqref{e:Qerror2}:
\begin{align*}
    &\int_{Q_{q,i}^{(error)}(t)}\min\{|B_{q+1}|^p+|E_{q+1}|^{p'},s\}\,dx+\\
    &\qquad+s\left|\{x\in Q^{(good)}_{q,i}(t):|B_{q+1}|^p+|E_{q+1}|^{p'}>s\}\right|\\
    &\leq \beta_{q+1}^{2(p+1)}|Q_{q,i}|\min(|B_{q,i}|^p+|E_{q,i}|^{p'},s)\,,\\
    &\int\int_{Q^{(error)}_{q,i}}1+|B_{q+1}|^p+|E_{q+1}|^{p'}\,dxdt\leq \frac{1}{2}(|B_{q,i}|^p+|E_{q,i}|^{p'})|Q_{q,i}|,
\end{align*}
where the latter is obtained by an appropriate choice of $\eps_{q+1}$ in \eqref{e:Qerror2} . We set
\begin{equation*}
    Q_{q+1}^{(error)}:=\bigcup_iQ_{q+1,i}^{(error)},\quad Q_{q+1}^{(good)}=Q_q^{(good)}\cup\bigcup_iQ_{q+1,i}^{(good)},\quad \mathcal{N}_{q+1}=\mathcal{N}_q\cup\bigcup_i\mathcal{N}_{q+1,i}.
\end{equation*}
Then we obtain
\begin{align*}
\int\int_{Q^{(error)}_{q+1}}|B_{q+1}|^p+|E_{q+1}|^{p'}\,dxdt& \leq \sum_i\int\int_{Q^{(error)}_{q+1,i}}|B_{q+1}|^p+|E_{q+1}|^{p'}\,dxdt\\
&\leq \frac12\sum_i|Q_{q,i}|(|B_{q,i}|^p+|E_{q,i}|^{p'})\\
&=\frac12\int \int_{Q_q^{(error)}}|B_q|^p+|E_q|^{p'}\,dxdt,
\end{align*}
so that \eqref{e:inductive1} is satisfied. To obtain \eqref{e:inductive2} we calculate
\begin{align*}
    \mathcal{I}_{q+1}(s,t)&=\sum_i\int_{Q_{q,i}^{(error)}(t)}\min\{|B_{q+1}|^p+|E_{q+1}|^{p'},s\}\,dx+\\
    &\qquad +\sum_is\left|\{x\in Q^{(good)}_{q,i}(t):|B_{q+1}|^p+|E_{q+1}|^{p'}>s\}\right|\\
    &\qquad +s\left|\{x\in Q^{(good)}_{q}(t):|B_{q}|^p+|E_{q}|^{p'}>s\}\right|\\
    &\leq \beta_{q+1}^{2(p+1)}\sum_i|Q_{q,i}|\min(|B_{q,i}|^p++|E_{q,i}|^{p'},s)\\
    &\qquad +s\left|\{x\in Q^{(good)}_{q}(t):|B_{q}|^p+|E_{q}|^{p'}>s\}\right|\\
    &=\beta_{q+1}^{2(p+1)}\int_{Q_{q}^{(error)}(t)}\min\{|B_{q}|^p+|E_{q}|^{p'},s\}\,dx+\\
    &\qquad +s\left|\{x\in Q^{(good)}_{q}(t):|B_{q}|^p+|E_{q}|^{p'}>s\}\right|\\
    &\leq \beta_{q+1}^{2(p+1)}\mathcal{I}_{q}(s,t).
\end{align*}
This completes the inductive step with estimates \eqref{e:inductive1}-\eqref{e:inductive2}.

Now observe that, because of \eqref{e:inductive1}, in particular $|Q_q^{(error)}|\to 0$ as $q\to \infty$. Since $(B_{q+1},E_{q+1})=(B_q,E_q)$ outside $Q_q^{(error)}$, we deduce that the sequence $(B_q,E_q)$ converges almost everywhere to \emph{piecewise constant} vector fields $(B,E)$. Furthermore, for any $q\in \N$, $s>1$ and $t$
\begin{equation*}
    \mathcal{I}_q(s,t)\leq \left(\prod_{k=1}^q\beta_k\right)^{2(p+1)} |Q(t)|\min(|B_0|^p+|E|_0|^{p'},s).
\end{equation*}
Thus, choosing $\beta_q>1$ in such a way that $\prod_{k=1}^q\beta_k^{2(p+1)}\leq 2$, we obtain the uniform bounds \eqref{e:weakLp}.
We note furthermore that, since $p,p'>1$, these estimates imply uniform $L^q$ bounds for some $q>1$ on both sequences $\{B_q\}$ and $\{E_q\}$, which, together with pointwise convergence implies strong $L^1$ convergence. Therefore \eqref{e:Faraday} and the boundary conditions given by 
$(B_0,E_0)$ in the sense of \eqref{e:boundarycondition} remain valid in the limit. This completes the proof.

\end{proof}


\section{Proof of Theorem \ref{t:main}} \label{s:Proof of Theorem 3}
Our next task is to prove Theorem \ref{t:main}. We divide the proof into several propositions and refer to~\cite{FLS} for many of the proofs. We will modify the piecewise constant map $\bar{V} \cong (0,0,\bar{B},\bar{E})$ in each space-time domain $Q_i$ separately and then superimpose the perturbations to get the solution $(u,B)$ whose existence is claimed in Theorem \ref{t:main}.

\subsection{Ensuring that $\bar{V}$ takes values in the hull}
Our first goal is to ensure that pointwise a.e., $\bar{V} = (0,0,\bar{B},\bar{E})$ belongs to the relative interior of the hull determined by $\zeta_+$ and $\zeta_-$. The precise statement is as follows:

\begin{proposition} \label{p:hull corollary}
We have $(0,0,\bar{B},\bar{E}) \in \mathcal{U}_{z_+(x,t),z_-(x,t)}$ for every $(x,t) \in \bar{Q}_i$ and every $i \in \N$.
\end{proposition}

In fact, we prove the following more general statement where one can set $\delta_0 = 1/M_0$, $r = \zeta_+$ and $s = \zeta_-$ to get Proposition \ref{p:hull corollary}:

\begin{proposition} \label{p:hull computation}
There exists $\delta_0 > 0$ with the following property: whenever $r, s > 0$, we have
\[\abs{u}^2 + \abs{B}^2 + \abs{S} + \abs{E} \le \delta _0 \min \{r^2,s^2\}, \; B \cdot E = 0 \quad \Longrightarrow \quad (u,B,S,E) \in \mathcal{U}_{r,s}.\]
\end{proposition}

We recall from~\cite{FLS} some notions that are relevant to the computations below. The relaxed Els\"{a}sser variables $(z^+,z^-,M) \in \R^3 \times \R^3 \times \R^{3 \times 3}$ are defined via $(u,S,B,E)$ as follows:
\[z^\pm \defeq u \pm b, \qquad M \xi \equiv S \xi + \xi \times E.\]
The wave cone conditions are written in terms of $(u,B,S,E)$ and $(z^+,z^-,M)$ as
\begin{align*}
& u \cdot \xi_x = B \cdot \xi_x = 0, \quad \xi_t u + S \xi_x = 0, \quad \xi_t B + \xi_x \times E = 0, \\
& z^\pm \cdot \xi = 0, \quad \xi_t z^+ + M \xi_x = 0, \quad \xi_t z^- + M^T \xi_x = 0.
\end{align*}
When $B \times u \neq 0$, it suffices to check the conditions
\begin{equation} \label{eq:easier wave cone condition}
S (B \times u) + (E \cdot u) u = 0, \qquad B \cdot E = 0.
\end{equation}
Recall that
\begin{align*}
& \mathscr{M} \defeq \{(u,B,S,E) \colon B \cdot E = 0\}, \\
& K \defeq \{(u,B,S,E) \colon E = B \times u, \, S = S_{u,B} + \Pi I, \, \Pi \in \R\}, \\
& K_{r,s} \defeq \{(u,B,S,E) \in K \colon |u+B| = r, \, |u-B| = s, \, \abs{\Pi} \le rs\}, \\
& \mathcal{U}_{r,s} \defeq \textup{int}_{\mathscr{M}}(K_{r,s}^{lc,\Lambda}),
\end{align*}
where $S_{u,B} \defeq u \otimes u - B \otimes B$.

We divide the proof of Proposition \ref{p:hull computation} into several steps. We refer to~\cite{FLS} when possible. The first lemma restates~\cite[Lemma 6.10]{FLS}.

\begin{lemma} \label{lem:Lemma 6.10}
If $\abs{z^+} \le r$, $\abs{z^-} \le s$ and $\abs{\Pi} \le rs$, then
\[(u,B,S_{u,B} + \Pi I, B \times u) \in K_{r,s}^{lc,\Lambda}.\]
\end{lemma}

Below, all the parameters $\delta_i > 0$ are independent of $r$ and $s$. The next lemma adds a primitive metric $e \otimes e$ to the $S$-component. It is proved by following the proof of~\cite[Lemma 6.11]{FLS} verbatim.

\begin{lemma} \label{lem:Lemma 6.11}
There exists $\delta_1 > 0$ such that if $|u|^2 + |B|^2 + |e|^2 + |\Pi| \le \delta_1 \min\{r^2,s^2\}$, then
\[(u,S_{u,B} + e \otimes e + \Pi I, B, B \times u) \in K_{r,s}^{lc,\Lambda}.\]
\end{lemma}

The next lemma replaces the primitive metric $e \otimes e$ by a general small symmetric matrix $S$. The lemma is proved by making obvious, slight modifications to the proof of~\cite[Lemma 6.12]{FLS}.

\begin{lemma} \label{lem:Lemma 6.12}
There exists $\delta_2 > 0$ such that whenever $r,s > 0$ and $\abs{u}^2 + \abs{B}^2 + \abs{S} + \abs{\Pi} \le \delta_2 \min\{r^2,s^2\}$, we have
\[(u,S_{u,B} + S + \Pi I,B,B \times u) \in K_{r,s}^{lc,\Lambda}.\]
\end{lemma}

We next get rid of the terms $S_{u,B}$ and $\Pi$.

\begin{corollary} \label{cor:general S}
There exists $\delta_3 > 0$ such that whenever $\abs{u}^2 + \abs{B}^2 + \abs{S} \le \delta_3 \min \{r^2,s^2\}$, we have
\[(u,S,B,B \times u) \in K_{r,s}^{lc,\Lambda}.\]
\end{corollary}

\begin{proof}
If $\abs{u}^2 + \abs{B}^2 + \abs{S} \le \delta_2 \min\{r^2/2,s^2/2\}$, we denote $\tilde{S} \defeq S - S_{u,b}$. Now $(u,S,B,B \times u) = (u,S_{u,B} + \tilde{S},B,B \times u)$ satisfies the assumptions of Lemma \ref{lem:Lemma 6.12}.
\end{proof}

In order to finish the proof of Proposition \ref{p:hull computation} we relax the condition $E = B \times u$ to $B \cdot E = 0$. This corresponds to~\cite[Lemma 6.13]{FLS}, but some of the details are different, and we therefore present a proof for completeness.

\begin{lemma} \label{lem:Lemma 6.13}
There exists $\delta_4 > 0$ such that whenever $\abs{u}^2 + \abs{B}^2 + \abs{S} + \abs{E} \le \delta_4 \min\{r^2,s^2\}$ and $B \cdot E = 0$, we have
\[(u,S,B,E) \in K_{r,s}^{lc,\Lambda}.\]
\end{lemma}

\begin{proof}
We first assume that $B \neq 0$ and that $\abs{u}^2 + \abs{B}^2 + \abs{S} \le \delta_3 \min\{r^2/2,s^2/2\}$. We start by considering $E = B \times u + B \times v$, where $0 < \abs{B \times v} \le \delta_3 \min \{r^2/2,s^2/2\}$; the proof of the case $B \neq 0$ is then finished by arguing as in Corollary \ref{cor:general S}. Without loss of generality, we assume that $B \cdot v = 0$. Then $\abs{B \times v} = \abs{B} \abs{v}$.

We denote $c \defeq (\abs{B}/\abs{v})^{1/2}$ so that 
\begin{equation}\label{small}\abs{c v} = |c^{-1} B| = (\abs{B} \abs{v})^{1/2} = \abs{B \times v}^{1/2} \le \delta_3 \min \{r^2/2,s^2/2\}.
\end{equation}
We then write $(u, S, B, B \times u + B \times v)$ as the middle point of a suitable $\Lambda$-segment:
\begin{align*}
  & (u,S,B, B \times (u+v)) \\
= & \frac{1}{2} (u+c v,S+\bar{S}, B + c^{-1} B, (1+c^{-1}) B \times (u + cv)) \\
+ & \frac{1}{2} (u-c v,S-\bar{S}, B - c^{-1} B, (1-c^{-1}) B \times (u - cv)),
\end{align*}
where
\[\bar{S} = \frac{c u \cdot B \times v}{\abs{B \times v}^2} (B \times v \otimes v + v \otimes B \times v).\]
Notice that thanks to \eqref{small} we can apply Corollary \ref{cor:general S} to deduce that the endpoints lie in $K_{r,s}^{lc,\Lambda}$. (In particular, $|\bar{S}| \le 2 c \abs{u} \abs{v} \le \delta_3 \min\{r^2/2,s^2/2\}$.) The direction of the $\Lambda$-segment is 
\[
\bar{V} = \left( 2 c v, 2 \bar{S}, 2 c^{-1} B, 2 \left( B \times c v + c^{-1} B \times u \right) \right),\]
which belongs to $\Lambda$ since \eqref{eq:easier wave cone condition} is satisfied. The case $B \neq 0$, $E = B \times u + B \times v$ is now proved. The general case $B \neq 0$, $E = B \times v$ is obtained as in the proof of Corollary \ref{cor:general S}.

Suppose then $B = 0$. Choose any tiny $\tilde{B} \neq 0$ with $\tilde{B} \cdot E = 0$. We may then write $E = \tilde{B} \times v$ as above. Now
\[(u,S,0,E)
= \frac{1}{2} (u,S,\tilde{B},\tilde{B} \times v)
+ \frac{1}{2} (u,S,-\tilde{B},-\tilde{B} \times (-v) \eqdef \frac{1}{2} V_1 + \frac{1}{2} V_2,\]
where $V_1,V_2 \in K_{r,s}^{lc,\Lambda}$ by the case $B \neq 0$ and $V_1-V_2 = (0,0,2\tilde{B},0) \in \Lambda$.
\end{proof}

\subsection{Modifying $\bar{V}$ in a single set $Q_i$} \label{s:Modifying in a single cube}
As the main building block of the proof of Theorem \ref{t:main}, we perturb $\bar{V} = (0,0,\bar{B},\bar{E})$ in a single set $Q_i$, where $\{Q_i\}_{i \in \N}$ is the family of disjoint open sets corresponding to $\bar{B},\bar{E}$ in the definition of piecewise constant fields.

\begin{proposition} \label{p:perturbation in a cube}
Under the assumptions of Theorem \ref{t:main}, fix $i \in \N$. There exists $V = (u,S,B,E)$ with $u,B \in L^\infty([0,T];L^2(\T^3))$ and $S,E \in L^\infty(0,T;L^1(\T^3))$ such that $V$ solves relaxed MHD equations in $\T^3$, $|B\pm u| = \zeta_\pm$ a.e. $(x,t) \in Q_i$, $V$ solves \eqref{e:MHD} in $Q_i$ and $V = \bar{V}$ a.e. outside $Q_i$.
\end{proposition}

Proposition \ref{p:perturbation in a cube} follows from slight modifications of the results of ~\cite[\textsection 7]{FLS}, but below, we indicate the main ideas. In order to keep the notation consistent with~\cite{FLS} we denote $\Omega = Q_i$.

By assumption, $\bar{E}$ and $\bar{B}$ are locally constant in $\bar{\Omega}$ with $\bar{B} \cdot \bar{E} = 0$. Furthermore, $\zeta_+,\zeta_- \in C(\bar{\Omega})$ satisfy
\[M_0 (|\bar{B}|^2+|\bar{E}|) \le \min\{\zeta_+^2,\zeta_-^2\} \quad \textup{for all } (x,t) \in \bar{\Omega}.\]
Following~\cite{FLS} we write $(\bar{B},\bar{E}) \cong \bar{\omega} = d \bar{\varphi} \wedge d \bar{\psi}$, where $\bar{\varphi}$ and $\bar{\psi}$ are linear. Our aim is to replace $\bar{V} = (0,0,\bar{\omega})$ in $\bar{\Omega}$ by $V = (u,S,\omega)$ which satisfies the ideal MHD equations and $\abs{B \pm u} = \zeta_\pm$ a.e. in $\Omega$.

In analogy to~\cite[Definition 7.1]{FLS} we define a class of subsolutions
\begin{align*}
X_0
&\defeq \{ (V = (u,S,\omega) \in C^\infty(\bar{\Omega},\R^{15}) \colon \textup{there exist } \varphi,\psi \in C^\infty(\bar{\Omega},\R^{15}) \textup{ such that } \\
& \omega = d\varphi \wedge d\psi, \, \mathcal{L}(V) = 0, \; \textup{supp}(u,S,\varphi-\bar{\varphi},\psi-\bar{\psi}) \subset \Omega \textup{ and } \\
& V(x,t) \in \mathcal{U}_{\zeta_+(x,t),\zeta_-(x,t)} \, \forall (x,t) \in \bar{\Omega}\}.
\end{align*}
With $C^\pm \defeq \max_{(x,t) \in \bar{Q}_i} \zeta_\pm(x,t)$ we denote the weak sequential closure of $X_0$ in $L^2(\T^3 \times [0,T];\overline{\textup{co}}(K_{C^+,C^-})$ by $X$. Now $X \ni \{\bar{V}\}$ is a compact metrisable space, and we denote a metric by $d_X$.

We state the main step of the proof of Proposition \ref{p:perturbation in a cube}.

\begin{proposition} \label{p:Proposition 7.2}
There exists $C > 0$ with the following property. If $V = (u,S,d \varphi \wedge d \psi) \in X_0$, then there exist $V_\ell = (u_\ell, S_\ell, d \varphi_\ell \wedge d \psi_\ell) \in X_0$ such that $d_X(V_\ell,V_0) \to 0$ and
\begin{align*}
&\int_\Omega (\abs{u_\ell(x,t)}^2 + \abs{B_\ell(x,t)}^2 - \abs{u(x,t)}^2 - \abs{B(x,t)}^2) \, dx \, dt \\
\ge \; & C \int_\Omega \left( \frac{\zeta_+(x,t)^2+\zeta_-(x,t)^2}{2} - \abs{u(x,t)}^2 - \abs{B(x,t)}^2 \right) \, dx \, dt.
\end{align*}
\end{proposition}

In fact,~\cite[Proposition 7.2]{FLS} is the special case of Proposition \ref{p:Proposition 7.2} where $\zeta_+$ and $\zeta_-$ are constant in $(x,t)$, and the proof of~\cite[Proposition 7.2]{FLS} applies with relatively minor changes. With Proposition \ref{p:Proposition 7.2} in hand, the proof of Proposition \ref{p:perturbation in a cube} is completed by standard methods, see~\cite[pp. 38--39]{FLS}.

\subsection{Adding the perturbations}
We finish the proof of Theorem \ref{t:main} by iterating Proposition \ref{p:perturbation in a cube} and ensuring that the magnetic helicity does not change along the iteration.

\begin{proposition}
There exist $u,B \in L^\infty([0,T];L^2(\T^3))$ that satisfy the assertions of Theorem \ref{t:main}.
\end{proposition}

\begin{proof}
Given $N \in \N$, we use Proposition \ref{p:perturbation in a cube} to get a solution $V_N = (u_N,S_N,B_N,E_N)$ of the relaxed MHD equations that solves \eqref{e:MHD} in $\cup_{i=1}^N Q_i$ and satisfies $V_N = \bar{V}$ a.e. outside $\cup_{i=1}^N Q_i$. The a.e. limit $V(x,t) = \lim_{N \to \infty} V_N(x,t)$ satisfies $\lim_{N \to \infty} \|u_N-u\|_{L^2} = \lim_{N \to \infty} \|B_N-B\|_{L^2} = 0$ and is a solution of \eqref{e:MHD} with $|B\pm u| = \zeta_\pm$ a.e. $(x,t) \in \T^3 \times [0,T]$. Thus $B,u \in L^\infty([0,T];L^2(\T^3))$.

We then show that $\mathcal{H}(B)(t) = \mathcal{H}(\bar{B})(t)$ a.e. $t \in (0,T)$. Since $B_N \to B$ in $L^2(\T^3 \times [0,T])$, by a standard argument it suffices to show that $\mathcal{H}(B_N)(t) = \mathcal{H}(\bar{B})(t) \text{ a.e. } t \in (0,T) \text{ for every }N \in \N$. In fact, by setting $\Omega = \cup_{i=1}^N Q_i$ and defining $X_0$ and $X$ as in \textsection \ref{s:Modifying in a single cube}, we approximate $B_N$ weakly in $L^2(\T^3 \times [0,T])$ by a sequence of subsolutions $B_N^k \in X_0$. Once we show that
\begin{equation} \label{e:magnetic helicity unchanged}
\mathcal{H}(B_N^k)(t) = \mathcal{H}(\bar{B})(t) \text{ a.e. } t \in (0,T) \text{ for every }N,k \in \N,
\end{equation}
we get $\mathcal{H}(B_N^k)(t) = \mathcal{H}(B_N)(t)$ a.e. $t \in (0,T)$ (see [FLS21, Theorem 2.2] and its proof).

Fix, therefore, $N,k \in \N$ and recall that for each $i \in \{1,\ldots,N\}$ we can write
\[B_N^k = \nabla \varphi_i \times \nabla \psi_i \text{ and } \bar{B} = \nabla \bar{\varphi}_i \times \nabla \bar{\psi}_i \quad \text{in } Q_i,\]
where $\varphi_i$ and $\psi_i$ are smooth, $\bar{\varphi}_i$ and $\bar{\psi}_i$ are linear and $\textup{supp} \, (\varphi_i-\bar{\varphi}_i,\psi_i-\bar{\psi}_i) \subset Q_i$. Given a vector potential $\bar{A}$ of $\bar{B}$, one vector potential of $B_N^k$ is therefore given by
\[A_N^k \defeq \bar{A} + \sum_{i=1}^N (\varphi_i \nabla \psi_i - \bar{\varphi}_i \nabla \bar{\psi}_i).\]

For a.e. $t \in (0,T)$ and each $i \in \{1,\ldots,N\}$ we denote $Q_i(t) \defeq \{x \in \T^3 \colon (x,t) \in Q_i\}$. By integrating by parts first on $\T^3$ and then on each $Q_i(t)$ we get
\begin{align*}
\int_{\T^3} A_N^k \cdot B_N^k
&= \int_{\T^3} \bar{A} \cdot \bar{B} + \int_{\T^3} (A_N^k-\bar{A}) \cdot (B_N^k + \bar{B}) \\
&= \int_{\T^3} \bar{A} \cdot \bar{B} + \sum_{i=1}^N \int_{Q_i(t)} (\varphi_i \nabla \psi_i - \bar{\varphi}_i \nabla \bar{\psi}_i) \cdot (\nabla \varphi_i \times \nabla \psi_i + \nabla \bar{\varphi}_i \times \nabla \bar{\psi}_i) \\
&= \int_{\T^3} \bar{A} \cdot \bar{B} + \sum_{i=1}^N \int_{Q_i(t)} (\varphi_i \nabla \psi_i \cdot \nabla \times (\bar{\varphi}_i \nabla \bar{\psi}_i) - \bar{\varphi}_i \nabla \bar{\psi}_i \cdot \nabla \times (\varphi_i \nabla \psi_i)) \\
&= \int_{\T^3} \bar{A} \cdot \bar{B},
\end{align*}
which proves \eqref{e:magnetic helicity unchanged} and completes the proof of Theorem \ref{t:main}.
\end{proof}

\section{Proof of Corollaries \ref{c:bounded} and \ref{c:L3weak}}
\label{s:corollaries}

\begin{proof}[Proof of Corollary \ref{c:bounded}]
We need different initial vector fields for the case $h>0$ and $h<0$. We select them uniformly for the sake of having unified constants. 

We first select piecewise constant vector fields $\bar{B}_1,\bar{B}_2 \in L^\infty(\T^3)$ such that $\textrm{div} \bar{B}_i = 0$ and $\mathcal{H}(\bar{B}_1) > 0 > \mathcal{H}(\bar{B}_2)$. This can be done e.g. by considering piecewise affine vector potentials $\bar{A} = (\bar{A}_1,\bar{A}_2,0) \in L^\infty(\T^3;\R^3)$ such that $\partial_3 \bar{A}_1 = 1$ in $\textrm{supp}(\bar{A}_2)$ with $\int_{\T^3} \bar{A}_2 \, dx > 0$ for $\bar{B}_1$ and $\int_{\T^3} \bar{A}_2 \, dx < 0$
for $\bar{B}_2$. We then fix $M > 0$ such that $2 M |\mathcal{H}(\bar{B}_i)| \geq M_0 \int_{\T^3}|\bar{B}_i(y)|^2 \, dy$ for $i = 1,2$, where $M_0 > 0$ is the geometric constant of Theorem \ref{t:main}.

Let $h > 0$ (The case $h = 0$ can be proven by minor modifications to~\cite{FLS}.) 

After scaling, we may assume that $\mathcal{H}(\bar{B}_1) = h$. In Theorem \ref{t:main}, set $\bar{B} = \bar{B}_1$, $\bar{E} = 0$ and $|\zeta_\pm(x,t)|^2 = M_0 |\bar{B}(x)|^2 + 2[e(t) \pm w(t)] - M_0 \int_{\T^3} |\bar{B}(y)|^2 \, dy$. Our assumptions give $2 [e(t) \pm w(t)] > 2M |h| \geq M_0 \int_{\T^3}|\bar{B}(y)|^2 \, dy$ for all $t \in [0,T]$, and therefore $M_0 (|\bar{B}|^2 + |\bar{E}|) \le \min\{\zeta_+^2,\zeta_-^2\}$ for a.e. $(x,t) \in \T^3 \times [0,T]$. Theorem \ref{t:main} now yields a weak solution $(u,B)\in L^{\infty}(\T^3\times [0,T])$ of \eqref{e:MHD} such that $\mathcal{E}(u,B)(t) = 4^{-1} \int_{\T^3} (|\zeta_+|^2 + |\zeta_-|^2) \, dx = e(t)$, $\mathcal{W}(u,B)(t) = 4^{-1} \int_{\T^3} (|\zeta_+|^2 - |\zeta_-|^2) \, dx = w(t)$ and $\mathcal{H}(B)(t) = h$ for a.e. $t$.
\end{proof}

\begin{proof}[Proof of Corollary \ref{c:L3weak}]
The proof follows immediately by combining Theorem \ref{t:Faraday} with $p=3$ and Theorem \ref{t:main}.
\end{proof}


\begin{thebibliography}{SK}






\bibitem[AFS08]{AFS}
Astala, K., Faraco, D., Sz\'{e}kelyhidi Jr, L.: 
Convex integration and the $L^p$ theory of elliptic equations,
Ann. Sc. Norm. Super. Pisa Cl. Sci. \textbf{7}, no. 1, 1--50 (2008)

\bibitem[AA96]{AA} Alberti, G., Ambrossio, L.: A geometric approach
to monotone functions in $\mathbb{R}^n$. Math. Z. no.230 259--316
(1996)

\bibitem[A86]{Arn}
Arnold, V.I.:
The asymptotic Hopf invariant and its applications.
Selecta Math. Soviet. \textbf{5}, no.~4, 327--345 (1986)

\bibitem[AK98]{ArnoldKhesinbook}
Arnold, V.I., Khesin, B.A.:
Topological methods in hydrodynamics.
Applied Mathematical Sciences, 125. Springer-Verlag, New York (1998)
\bibitem[BBV20]{BBV}
Beekie, R., Buckmaster, T., Vicol, V.:
Weak solutions of ideal MHD which do not conserve magnetic helicity.
Ann. PDE \textbf{6}, no.~1, 40 pp. (2020)



\bibitem[B84]{Berger}
Berger, M.A.:
Rigorous new limits on magnetic helicity dissipation in the solar corona.
Geophys. Astrophys. Fluid Dyn. \textbf{30}, no.~1--2, 79--104 (1984)



\bibitem[BSV13]{BSV}
Boros, N. Sz\'{e}kelyhidi, L, Jr., Volberg, A.:
Laminates meet Burkholder functions.
J. Math. Pures Appl. (9) \textbf{100}, no.~5, 687--700 (2013)

\bibitem[BLL15]{BLFNL}
Bronzi, A. C., Lopes Filho, M. C., Nussenzveig Lopes, H. J.:
Wild solutions for 2{D} incompressible ideal flow with passive tracer.
Commun. Math. Sci. \textbf{13}, 1333--1343 (2015)








\bibitem[BV21]{BVreview}
Buckmaster, T., Vicol, V.:
Convex integration constructions in hydrodynamics.
Bull. Amer. Math. Soc. \textbf{58}, no.~1, 44 pp. (2021)

\bibitem[CKS97]{CKS}
Caflisch, R.E., Klapper, I., Steele, G.:
Remarks on singularities, dimension and energy dissipation for ideal hydrodynamics and MHD.
Comm. Math. Phys. \textbf{184}, no.~2, 443--455 (1997) 


\bibitem[CCFS08]{CCFS}
Cheskidov, A., Constantin, P., Friedlander, S., Shvydkoy, R.:
Energy conservation and Onsager's conjecture for the Euler equations.
Nonlinearity \textbf{21}, no.~6, 1233--1252 (2008)




\bibitem [CFM05]{CFM}
Conti, S.; Faraco, D.; Maggi, F.:
A new approach to counterexamples to $L^1$ estimates: Korn's inequality, geometric rigidity, and regularity for gradients of separately convex functions.
Arch. Ration. Mech. Anal. \textbf{175}, no.~2, 287-–300 (2005)




\bibitem[CET94]{CET}
Constantin, P., E, W., Titi, E.S.:
Onsager's conjecture on the energy conservation for solutions of Euler's equation.
Comm. Math. Phys. \textbf{165}, no.~1, 207-209 (1994)







\bibitem[DLS09]{DLS09}
De Lellis, C., Sz\'{e}kelyhidi Jr., L.:
The Euler equations as a differential inclusion.
Ann. of Math. (2) \textbf{170}, no.~3, 1417--1436 (2009)








\bibitem[E15]{Eyi3}
Eyink G. L.:
Turbulent general magnetic reconnection.
The astrophysical Journal \textbf{807}:137, 209 pp. (2015)

\bibitem[F03]{F03} Faraco, D.:
Milton's conjecture on the regularity of solutions to isotropic equations.
Ann. Inst. H. Poincar\'{e} Anal. Non Lin\'{e}aire \textbf{20}, no.~5, 889–-909 (2003)
 
\bibitem[F04]{F04}
Faraco, D.:
Tartar conjecture and Beltrami operators.
Michigan Math. J. \textbf{52}, no.~1, 83–-104 (2004)

\bibitem[FL20]{FL}
Faraco, D., Lindberg, S.:
Proof of Taylor's conjecture on magnetic helicity conservation.
Comm. Math. Phys. \textbf{373}, no.~2, 707--738 (2020)

\bibitem[FLMV21]{FLMV}
Faraco, D., Lindberg, S., MacTaggart, D., Valli, A.:
On the proof of Taylor’s conjecture in multiply connected domains.
Applied Mathematics Letters (2021)

\bibitem[FLS21]{FLS}
Faraco, Daniel., Lindberg, S., Sz\'{e}kelyhidi Jr., L.:
Bounded solutions of ideal MHD with compact support in space-time.
Arch. Ration. Mech. Anal. \textbf{239} no.~1, 51--93 (2021)



\bibitem [FMO18]{FMO18} Faraco, D.,  Mora-Corral, C, and De la Oliva, M:
Sobolev homeomorphisms with gradients of low rank via laminates.
Adv. Calc. Var.  \textbf{11}, no.~2, 118--138 (2018)


\bibitem[GLBL06]{GLBL}
Gerbeau, J.-F., Le Bris, C., Leli\`{e}vre, T.:
Mathematical methods for the magnetohydrodynamics of liquid metals.
Numerical Mathematics and Scientific Computation, Oxford University Press, Oxford, (2006)


\bibitem[Has85]{Has85}
Hasegawa, A.:
Self-organization processes in continuous media.
Advances in Physics \textbf{34}, no. 1, 1--42. (1985)



\bibitem [H11]{H11} Hencl, S. Sobolev homeomorphism with zero Jacobian almost everywhere. J. Math. Pures Appl. (9) \textbf{95}, no.~4, 444--458 (2011)
 




\bibitem[KL07]{KL}
Kang, E., Lee, J.:
Remarks on the magnetic helicity and energy conservation for ideal magneto-hydrodynamics.
Nonlinearity \textbf{20}, no.~11, 2681--2689 (2007) 



\bibitem[K03]{Kir}
Kirchheim, B.:
Rigidity and geometry of microstructures.
Habilitation thesis, Universit\"{a}t Leipzig (2003)

\bibitem[LA91]{LA}
Laurence, P., Avellaneda, M.:
On Woltjer's variational principle for force-free fields.
J. Math. Phys. \textbf{32}, no.~5, 1240--1253 (1991)


\bibitem[LM16]{LM16} Liu, Z.M., Malý, J.:
A strictly convex Sobolev function with null Hessian minors.
Calc. Var. Partial Differential Equations \textbf{55}, no.~3, Art. 58, 19 pp. (2016)



\bibitem[MV19]{MV}
MacTaggart, D., Valli, A.:
Magnetic helicity in multiply connected domains.
J. Plasma Phys. \textbf{85}, 15 pp. (2019)




\bibitem[M69]{Mof}
Moffatt, H.K.:
Degree of knottedness of tangled vortex lines.
J. Fluid Mech. \textbf{35}, no.~1, 1187--129 (1969)




%
%

\bibitem[On49]{Onsager}
Onsager, L.:
Statistical hydrodynamics.
Nuovo Cimento (9) \textbf{6}. Supplemento, no. 2 (Convegno Internazionale di Meccanica Statistica), 279--287 (1949)

\bibitem[OS93]{OS}
Ortolani, S., Schnack, D.D.:
Magnetohydrodynamics of Plasma Relaxation.
World Scientific, Singapore (1993)




\bibitem[ST83]{ST}
Sermange, M., Temam, R.:
Some mathematical questions related to the MHD equations.
Comm. Pure Appl. Math. \textbf{36}, no.~5, 635--664 (1983)






\bibitem[Tar77]{Tartar}
Tartar, L., 
The compensated compactness method applied to systems of conservation laws.
In Systems of nonlinear partial differential equations. Dordrecht, 263-285 (1977)

\bibitem[Tar05]{Tartar2}
Tartar, L.: Compensation effects in partial differential equations.
Rend. Accad. Naz. Sci. XL Mem. Mat. Appl. (5), \textbf{29}, 395-453 (2005)



\bibitem[T74]{Taylor}
Taylor, J.B.:
Relaxation of toroidal plasma and generation of reverse magnetic fields.
Phys. Rev. Lett. \textbf{33}, 1139--1141 (1974)


\bibitem[W58]{Woltjer}
Woltjer, L.:
A theorem on force-free magnetic fields.
Proc. Natl. Aca. Sci. USA \textbf{44}, no.~6 (1958)



\bibitem[Y09]{Yu}
Yu, X.:
A note on the energy conservation of the ideal MHD equations.
Nonlinearity \textbf{22}, no.~4, 913--922 (2009)
\end{thebibliography}
\end{document}